\documentclass{amsart}
\usepackage{amsmath}
\usepackage{amssymb}
\usepackage{amsfonts}
\usepackage{graphicx}
\usepackage{amscd}

\newtheorem{theorem}{Theorem}
\theoremstyle{plain}

\newtheorem{definition}{Definition}
\newtheorem{example}{Example}

\newtheorem{proposition}{Proposition}

\numberwithin{equation}{section}

\input{tcilatex}

\begin{document}

\begin{center}
\bigskip {\Huge A New View on Soft Normed Spaces}

\textbf{\ Tunay BILGIN$^{\mathrm{a}}$, Sadi BAYRAMOV\textbf{$^{\mathrm{b}}$, 
}\c{C}igdem Gunduz(ARAS)\textbf{$^{\mathrm{c}}$, Murat Ibrahim YAZAR$^{%
\mathrm{b}}$}}

\medskip

$^{\mathrm{a}}$\textit{Department of Mathematics,}\\[0pt]
Yuzuncu Y\i l University\textit{$,$Van$,$Turkey}\\[0pt]

$^{\mathrm{b}}$\textit{Department of Mathematics$,$ Faculty of Science and
Letters,}\\[0pt]
\textit{Kafkas University$,$ TR-}$36100$\textit{\ Kars$,$ Turkey}\\[0pt]

$^{\mathrm{c}}$\textit{Department of Mathematics$,$ Faculty of Science and
Letters,}\\[0pt]
\textit{Kocaeli University$,$ TR-}$36100$\textit{\ Kocaeli$,$ Turkey}\\[0pt]

\textbf{e-mails: tbilgin@yyu.edu.tr, baysadi@gmail.com,
carasgunduz@gmail.com, miy248@yahoo.com \ }

\bigskip

\textbf{Abstract}
\end{center}

\begin{quotation}
In this paper, we work on the structure of soft linear spaces over a field K
and investigate some of its properties. Here, we use the concept of the soft
point which was introduced in \cite{bayram,sam2}. We then introduce the soft
norm in soft linear spaces. Finally, we examine the properties of this soft
normed space and present some investigations about soft continuous operators
in the space.\textbf{{\large {\ }}}
\end{quotation}


\section{\protect\bigskip INTRODUCTION}

Molodtsov \cite{molod} introduced the notion of soft set to overcome
uncertainties which cannot be dealth with by classical methods in many areas
such as environmental science, economics, medical science, engineering and
social sciences. This theory is applicable where there is no clearly defined
mathematical model. Recently, many papers concerning soft sets have been
published; see [1-7]

The concept of soft point was defined in different approaches. Among these,
the soft point given in \cite{bayram,sam2} is more accurate. Also in the
study \cite{sam2}, S.Das and et all introduced the concept of soft metric
and investigated some properties of soft metric spaces.

Because of the difficulties to define a vector space over a soft set based
upon the concept of soft point S.Das and et all introduced the concept of
soft element in \cite{element} and defined a soft vector space by using the
concept of soft element. After then they studied on soft normed spaces, soft
linear operators, soft inner product spaces and their basic properties \cite%
{donlin, dlin, inner}.

In this paper, by using the concept of soft point we define the soft vector
space in a new point of view and investigate some of its properties. We then
introduce the soft norm in soft vector spaces. Finally, we examine the
properties of the soft normed space and present some investigations about
soft continuous linear operators in the space.

\section{\textbf{PRELIMINARIES}}

\begin{definition}
(\cite{molod}) A pair $(F,E)$ is called a soft set over $X$ $,$where $F$ is
a mapping given by $F:E\rightarrow P(X).$

\begin{definition}
(\cite{maji2}) A soft set $(F,E)$ over $X$ is said to be a null soft set
denoted by $\Phi ,$ if for all $e\in E,$ $F(e)=\phi $ \ (null set).
\end{definition}
\end{definition}

\begin{definition}
(\cite{maji2}) A soft set $(F,E)$ over $X$ is said to be an absolute soft
set denoted by $\tilde{X}$, if for all $\varepsilon \in E,$ $F(e)=X$.
\end{definition}

\begin{definition}
(\cite{sam1}) Let $%
%TCIMACRO{\U{211d} }%
%BeginExpansion
\mathbb{R}
%EndExpansion
$ be the set of real numbers and $B(%
%TCIMACRO{\U{211d} }%
%BeginExpansion
\mathbb{R}
%EndExpansion
)$ be the collection of all non-empty bounded subsets of $%
%TCIMACRO{\U{211d} }%
%BeginExpansion
\mathbb{R}
%EndExpansion
$ and $E$ taken as a set of parameters. Then a mapping $F:E\rightarrow B(%
%TCIMACRO{\U{211d} }%
%BeginExpansion
\mathbb{R}
%EndExpansion
)$ is called a soft real set. If a soft real set is a singleton soft set, it
will be called a soft real number and denoted by $\tilde{r},$ $\tilde{s},$ $%
\tilde{t}$ etc.
\end{definition}

$\tilde{0},$ $\tilde{1}$ are the soft real numbers where $\tilde{0}(e)=0,$ $%
\tilde{1}(e)=1$ for all $e\in E$ , respectively.

\begin{definition}
(\cite{sam1}) Let $\tilde{r},$ $\tilde{s}$ be two soft real numbers. then
the following statements
\end{definition}

\begin{enumerate}
\item[(i)] \qquad\ $\tilde{r}\tilde{\leq}\tilde{s}$ if \ $\tilde{r}(e)\leq $ 
$\tilde{s}(e),$ for all $e\in E$ ;

\item[(ii)] \qquad $\tilde{r}\tilde{\geq}\widetilde{s}$ if \ $\tilde{r}%
(e)\geq $ $\tilde{s}(e),$ for all $e\in E$ ;

\item[(iii)] \qquad $\tilde{r}\tilde{<}$ $\widetilde{s}$ if \ $\tilde{r}(e)<$
$\tilde{s}(e),$ for all $e\in E$ ;

\item[(iv)] \qquad $\tilde{r}\tilde{>}$ $\tilde{s}$ if \ $\tilde{r}(e)>$ $%
\tilde{s}(e),$ for all $e\in E$ ;

hold.
\end{enumerate}

\begin{definition}
(\cite{bayram, sam2}) A soft set $(F,E)$ over $X$ is said to be a soft point
if there is exactly one $e\in E$, such that $F(e)=\left\{ x\right\} $ for
some $x\in X$ and $F(e^{\prime })=\varnothing ,$ $\forall e^{\prime }\in
E/\left\{ e\right\} $. It will be denoted by $\tilde{x}_{e}.$
\end{definition}

\begin{definition}
(\cite{bayram, sam2}) Two soft point $\tilde{x}_{e},$ $\tilde{y}_{e^{\prime
}}$ are said to be equal if $e=e^{\prime }$ and $x=y$. Thus $\tilde{x}%
_{e}\neq $ $\tilde{y}_{e^{\prime }}$ $\Leftrightarrow x\neq y$ or $e\neq
e^{\prime }$.
\end{definition}

\begin{proposition}
(\cite{bayram}) Every soft set can be expressed as a union of all soft
points belonging to it. Conversely, any set of soft points can be considered
as a soft set.
\end{proposition}

\textit{Let }$SP(\tilde{X})$\textit{\ be the collection of all soft points
of }$\tilde{X}$\textit{\ and }$%
%TCIMACRO{\U{211d} }%
%BeginExpansion
\mathbb{R}
%EndExpansion
(E)^{\ast }$\textit{\ denote the set of all non-negative soft real numbers.}

\begin{definition}
(\cite{sam2}) A mapping $\tilde{d}:SP(\tilde{X})\times SP(\tilde{X}%
)\rightarrow 
%TCIMACRO{\U{211d} }%
%BeginExpansion
\mathbb{R}
%EndExpansion
(E)^{\ast }$ is said to be a soft metric on the soft set $\tilde{X}$ if $%
\tilde{d}$ satisfies the following conditions:
\end{definition}

\begin{enumerate}
\item[(M1)] $\tilde{d}(\tilde{x}_{e_{1}},\tilde{y}_{e_{2}})\tilde{\geq}%
\tilde{0}$ \textit{for all }$\tilde{x}_{e_{1}},\tilde{y}_{e_{2}}\tilde{\in}%
\tilde{X},$

\item[(M2)] $\tilde{d}(\tilde{x}_{e_{1}},\tilde{y}_{e_{2}})=\tilde{0}$%
\textit{\ if and only if }$\tilde{x}_{e_{1}}=\tilde{y}_{e_{2}}\tilde{\in}%
\tilde{X},$

\item[(M3)] $\tilde{d}(\tilde{x}_{e_{1}},\tilde{y}_{e_{2}})=\tilde{d}(\tilde{%
y}_{e_{2}},\tilde{x}_{e_{1}})$\textit{\ for all }$\tilde{x}_{e_{1}},\tilde{y}%
_{e_{2}}\tilde{\in}\tilde{X},$

\item[(M4)] \textit{For all }$\tilde{x}_{e_{1}},\tilde{y}_{e_{2}},\tilde{z}%
_{e_{3}}\tilde{\in}\tilde{X},$\textit{\ }$\tilde{d}(\tilde{x}_{e_{1}},\tilde{%
z}_{e_{3}})\tilde{\leq}$\textit{\ }$\tilde{d}(\tilde{x}_{e_{1}},\tilde{y}%
_{e_{2}})+\tilde{d}(\tilde{y}_{e_{2}},\tilde{z}_{e_{3}}).$
\end{enumerate}

\textit{The soft set }$\tilde{X}$\textit{\ with a soft metric }$\tilde{d}$%
\textit{\ is called a soft metric space and denoted by }$(\tilde{X},\tilde{d}%
,E)$\textit{.}

\section{Soft Normed Linear Spaces}

In this section, by using the concept of soft point we define the soft
vector space and soft norm in a new point of view and investigate the
properties of the soft normed space.

Let $X$ \ be a vector space over a field $K$ $(K=%
%TCIMACRO{\U{211d} }%
%BeginExpansion
\mathbb{R}
%EndExpansion
)$ and the parameter set $E$ be the real number set $%
%TCIMACRO{\U{211d} }%
%BeginExpansion
\mathbb{R}
%EndExpansion
.$

\begin{definition}
Let $(F,E)$ be a soft set over $X.$ The soft set $(F,E)$ is said to be a
soft vector and denoted by $\tilde{x}_{e}$ if there is exactly one $e\in E$,
such that $F(e)=\left\{ x\right\} $ for some $x\in X$ and $F(e^{\prime
})=\varnothing ,$ $\forall e^{\prime }\in E/\left\{ e\right\} .$
\end{definition}

The set of all soft vectors over $\tilde{X}$ will be denoted by $SV(\tilde{X}%
).$

\begin{proposition}
The set $SV(\tilde{X})$ is a vector space according to the following
operations;
\end{proposition}

\begin{enumerate}
\item $\tilde{x}_{e}+\tilde{y}_{e^{\prime }}=(\widetilde{x+y})_{(e+e^{\prime
})}$\textit{\ \ for every }$\tilde{x}_{e},\tilde{y}_{e^{\prime }}\in SV(%
\tilde{X});$

\item $\tilde{r}.\tilde{x}_{e}=(\widetilde{rx})_{(re)}$\textit{\ for every }$%
\tilde{x}_{e}\in SV(\tilde{X})$\textit{\ and for every soft real number }$%
\tilde{r}.$
\end{enumerate}

\begin{definition}
\begin{proof}
If $\theta \in X$ is a zero vector and $e=0\in 
%TCIMACRO{\U{211d} }%
%BeginExpansion
\mathbb{R}
%EndExpansion
$ then $\tilde{\theta}_{0}$ is a soft zero vector in $SV(\tilde{X}).$
Furthermore, $(\widetilde{-x})_{(-e)}$ is the inverse of the soft vector $%
\tilde{x}_{e}.$ It is easy to see that the set $SV(\tilde{X})$ is a vector
space .
\end{proof}
\end{definition}

\begin{definition}
The set $SV(\tilde{X})$ is called soft vector space.
\end{definition}

\begin{definition}
A set $S=\left\{ \tilde{x}_{e_{1}}^{1},\tilde{x}_{e_{2}}^{2},...,\tilde{x}%
_{e_{n}}^{n}\right\} $ of soft vectors in $SV(\tilde{X})$ is said to be
linearly independent if the following condition 
\begin{equation*}
\tilde{r}_{1.}\tilde{x}_{e_{1}}^{1}+\tilde{r}_{2}.\tilde{x}_{e_{2}}^{2}+...+%
\tilde{r}_{n}.\tilde{x}_{e_{n}}^{n}=\tilde{\theta}_{0}\Leftrightarrow \tilde{%
r}_{1},\tilde{r}_{2},...,\tilde{r}_{n}=0.
\end{equation*}%
is satisfied for the soft real numbers $\tilde{r}_{i},$ $1\leq i\leq n.$
\end{definition}

\begin{proposition}
A set $S=\left\{ \tilde{x}_{e_{1}}^{1},\tilde{x}_{e_{2}}^{2},...,\tilde{x}%
_{e_{n}}^{n}\right\} $ of soft vectors in $SV(\tilde{X})$ $\ $is linearly
independent if the elements of the set $\left\{
x^{1},x^{2},...,x^{n}\right\} $ in $X$ are linearly independent and the
condition $(r_{1.}e_{1}+r_{2.}e_{2}+...+r_{n.}e_{n})=0$ \ is satisfied.
\end{proposition}

\begin{proof}
For any soft real number $\tilde{r}_{i},$ $1\leq i\leq n$ \ 
\begin{equation*}
\tilde{r}_{1.}\tilde{x}_{e_{1}}^{1}+\tilde{r}_{2}.\tilde{x}_{e_{2}}^{2}+...+%
\tilde{r}_{n}.\tilde{x}_{e_{n}}^{n}=\tilde{\theta}_{0}
\end{equation*}%
\begin{eqnarray*}
&\Leftrightarrow &(\widetilde{r_{1.}x^{1}+r_{2.}x^{2}+...+r_{n.}x^{n}}%
)_{(r_{1.}e_{1}+r_{2.}e_{2}+...+r_{n.}e_{n})}=\tilde{\theta}_{0} \\
&\Leftrightarrow &(\widetilde{r_{1.}x^{1}+r_{2.}x^{2}+...+r_{n.}x^{n}})=%
\tilde{\theta}\text{ and }(r_{1.}e_{1}+r_{2.}e_{2}+...+r_{n.}e_{n})=0 \\
&\Leftrightarrow &\tilde{r}_{1},\tilde{r}_{2},...,\tilde{r}_{n}=\tilde{0}.
\end{eqnarray*}
\end{proof}

\begin{definition}
Let $SV(\tilde{X})$ be a soft vector space and $\tilde{M}\subset SV(\tilde{X}%
)$ \ be a subset. If $\tilde{M}$ is a soft vector space, then $\tilde{M}$ is
said to be a soft vector subspace of $SV(\tilde{X})$ and denoted by $SV(%
\tilde{M})\tilde{\subset}SV(\tilde{X}).$
\end{definition}

\begin{example}
Let us given $\ $a class of soft vectors $\left\{ \tilde{x}%
_{e_{k}}^{k}\right\} _{k=\overline{1,n}}$ . Then the space $\left\{ \underset%
{i=1}{\overset{n}{\dsum }}\tilde{r}_{i}.\tilde{x}_{e_{i}}^{i}\right\} $
generated by the class $\left\{ \tilde{x}_{e_{k}}^{k}\right\} _{k=\overline{%
1,n}}$ is a soft vector subspace.
\end{example}

\begin{example}
If $M\subset X$ is a vector subspace then $SV(\tilde{M})\subset SV(\tilde{X}%
) $ is a soft vector subspace.
\end{example}

By using the definition of the soft vector, we can give the natural
definition of soft norm as follows.

\begin{definition}
\ Let $SV(\tilde{X})$ be a soft vector space. Then a mapping 
\begin{equation*}
\left\Vert .\right\Vert :SV(\tilde{X})\rightarrow 
%TCIMACRO{\U{211d} }%
%BeginExpansion
\mathbb{R}
%EndExpansion
^{+}(E)
\end{equation*}%
is said to be a soft norm on $SV(\tilde{X}),$ if \ $\left\Vert .\right\Vert $%
satisfies the following conditions:
\end{definition}

\begin{enumerate}
\item[(N1).] $\left\Vert \tilde{x}_{e}\right\Vert \tilde{\geq}\tilde{0}$ 
\textit{for all } $\tilde{x}_{e}\tilde{\in}SV(\tilde{X})$\textit{\ } and $%
\left\Vert \tilde{x}_{e}\right\Vert =\tilde{0}\Leftrightarrow \tilde{x}_{e}=%
\tilde{\theta}_{0};$

\item[(N2).] $\left\Vert \tilde{r}\text{.}\tilde{x}_{e}\right\Vert
=\left\vert \tilde{r}\right\vert \left\Vert \tilde{x}_{e}\right\Vert $ \ 
\textit{for all} $\tilde{x}_{e}\tilde{\in}SV(\tilde{X})$ \textit{and for
every soft scalar} $\tilde{r};$

\item[(N3).] $\left\Vert \tilde{x}_{e}+\tilde{y}_{e^{\prime }}\right\Vert 
\tilde{\leq}\left\Vert \tilde{x}_{e}\right\Vert +\left\Vert \tilde{y}%
_{e^{\prime }}\right\Vert $ for all $\ \tilde{x}_{e},\tilde{y}_{e^{\prime }}%
\tilde{\in}SV(\tilde{X}).$
\end{enumerate}

\textit{The soft vector space }$SV(\tilde{X})$ \textit{with a soft norm }$%
\left\Vert .\right\Vert $\textit{\ on }$\tilde{X}$\textit{\ \ is said to be
a soft normed linear space and is denoted by }$(\tilde{X},\left\Vert
.\right\Vert )$.

\begin{example}
Let $X$ be \ a normed space$.$ In this case, for every $\tilde{x}_{e}\tilde{%
\in}SV(\tilde{X}),$ 
\begin{equation*}
\left\Vert \tilde{x}_{e}\right\Vert =\left\vert e\right\vert +\left\Vert
x\right\Vert
\end{equation*}%
is a soft norm.
\end{example}

\textit{For every }$\tilde{x}_{e},\tilde{y}_{e^{\prime }}\tilde{\in}SV(%
\tilde{X})$\textit{\ and} \textit{for every soft scalar} $\tilde{r};$

\begin{enumerate}
\item[(N1).] $\left\Vert \tilde{x}_{e}\right\Vert =\left\vert e\right\vert
+\left\Vert x\right\Vert \tilde{\geq}\tilde{0},$%
\begin{equation*}
\left\Vert \tilde{x}_{e}\right\Vert =\tilde{0}\Leftrightarrow \left\vert
e\right\vert +\left\Vert x\right\Vert =\tilde{0}\Leftrightarrow e=0,\text{ }%
x=\tilde{\theta}\Leftrightarrow \tilde{x}_{e}=\tilde{\theta}_{0}
\end{equation*}

\item[(N2).] $\left\Vert \tilde{r}\text{.}\tilde{x}_{e}\right\Vert
=\left\Vert \widetilde{(r.x)}_{re}\right\Vert =\left\vert re\right\vert
+\left\Vert r.x\right\Vert =\left\vert r\right\vert \left( \left\vert
e\right\vert +\left\Vert x\right\Vert \right) =\left\vert \tilde{r}%
\right\vert \left\Vert \tilde{x}_{e}\right\Vert .$

\item[(N3).] 
\begin{eqnarray*}
\left\Vert \tilde{x}_{e}+\tilde{y}_{e^{\prime }}\right\Vert &=&\left\Vert (%
\widetilde{x+y})_{(e+e^{\prime })}\right\Vert =\left\vert e+e^{\prime
}\right\vert +\left\Vert x+y\right\Vert \\
&\leq &\left\vert e\right\vert +\left\vert e^{\prime }\right\vert
+\left\Vert x\right\Vert +\left\Vert y\right\Vert \\
&=&(\left\vert e\right\vert +\left\Vert x\right\Vert )+\left( \left\vert
e^{\prime }\right\vert +\left\Vert y\right\Vert \right) \\
&=&\left\Vert \tilde{x}_{e}\right\Vert +\left\Vert \tilde{y}_{e^{\prime
}}\right\Vert .
\end{eqnarray*}
\end{enumerate}

\begin{definition}
A sequence of soft vectors $\left\{ \tilde{x}_{e_{n}}^{n}\right\} $ in $(%
\tilde{X},\left\Vert .\right\Vert )$ is said to be convergent \ to $\tilde{x}%
_{e_{0}}^{0}$ ,if $\underset{n\rightarrow \infty }{lim}\left\Vert \tilde{x}%
_{e_{n}}^{n}-\tilde{x}_{e_{0}}^{0}\right\Vert =\tilde{0}$ and denoted by $%
\tilde{x}_{e_{n}}^{n}\rightarrow \tilde{x}_{\lambda _{0}}^{0}$ as $%
n\rightarrow \infty $.
\end{definition}

\begin{definition}
A sequence of soft vectors $\left\{ \tilde{x}_{e_{n}}^{n}\right\} $ in $(%
\tilde{X},\left\Vert .\right\Vert )$ is said to be a Cauchy sequence if
corresponding to every $\tilde{\varepsilon}\tilde{>}\tilde{0}$ , $\exists
m\in N$ such that $\left\Vert \tilde{x}_{e_{i}}^{i}-\tilde{x}%
_{e_{j}}^{j}\right\Vert \tilde{<}\tilde{\varepsilon},$ $\forall i,j\geq m$
i.e., $\left\Vert \tilde{x}_{e_{i}}^{i}-\tilde{x}_{e_{j}}^{j}\right\Vert
\rightarrow \tilde{0}$ as $i,j\rightarrow \infty .$
\end{definition}

\begin{proposition}
Every convergent sequence is a Cauchy sequence.
\end{proposition}

The proof is straight forward.

\begin{definition}
Let $(\tilde{X},\left\Vert .\right\Vert )$ be a soft normed linear space.
Then $(\tilde{X},\left\Vert .\right\Vert )$ is said to be complete if every
Cauchy sequence in $\tilde{X}$ converges to a soft vector of $\tilde{X}$.
\end{definition}

\begin{definition}
Every complete soft normed linear space is called a soft Banach space.
\end{definition}

\begin{proposition}
Every soft normed space is a soft metric space.
\end{proposition}

\begin{proof}
Let $(\tilde{X},\left\Vert .\right\Vert )$ be a soft normed space. If we
define the soft metric by $\tilde{d}(\tilde{x}_{e},\tilde{y}_{e^{\prime
}})=\left\Vert \tilde{x}_{e}-\tilde{y}_{e^{\prime }}\right\Vert $ for every $%
\tilde{x}_{e},\tilde{y}_{e^{\prime }}\tilde{\in}SV(\tilde{X})$ then it is
clear to show that the soft metric axioms are satisfied.
\end{proof}

\begin{theorem}
Let $\tilde{d}:SV(\tilde{X})\times SV(\tilde{X})\rightarrow 
%TCIMACRO{\U{211d} }%
%BeginExpansion
\mathbb{R}
%EndExpansion
^{+}(E)$ be \ a soft metric $.$ $SV(\tilde{X})$ is a soft normed space if
and only if \ the following conditions;
\end{theorem}

\begin{enumerate}
\item[a)] $\tilde{d}(\tilde{x}_{e}+\tilde{z}_{e^{\prime \prime }},\tilde{y}%
_{e^{\prime }}+\tilde{z}_{e^{\prime \prime }})=\tilde{d}(\tilde{x}_{e},%
\tilde{y}_{e^{\prime }})$

\item[b)] $\tilde{d}(\tilde{r}.\tilde{x}_{e},\tilde{r}.\tilde{y}_{e^{\prime
}})=\left\vert \tilde{r}\right\vert \tilde{d}(\tilde{x}_{e},\tilde{y}%
_{e^{\prime }})$
\end{enumerate}

satisfied.

\begin{proof}
If $\tilde{d}(\tilde{x}_{e},\tilde{y}_{e^{\prime }})=\left\Vert \tilde{x}%
_{e}-\tilde{y}_{e^{\prime }}\right\Vert ,$ then 
\begin{equation*}
\tilde{d}(\tilde{x}_{e}+\tilde{z}_{e^{\prime \prime }},\tilde{y}_{e^{\prime
}}+\tilde{z}_{e^{\prime \prime }})=\left\Vert \tilde{x}_{e}+\tilde{z}%
_{e^{\prime \prime }}-\tilde{y}_{e^{\prime }}-\tilde{z}_{e^{\prime \prime
}}\right\Vert =\left\Vert \tilde{x}_{e}-\tilde{y}_{e^{\prime }}\right\Vert =%
\tilde{d}(\tilde{x}_{e},\tilde{y}_{e^{\prime }})
\end{equation*}%
and%
\begin{equation*}
\tilde{d}(\tilde{r}.\tilde{x}_{e},\tilde{r}.\tilde{y}_{e^{\prime
}})=\left\Vert \tilde{r}\tilde{x}_{e}-\tilde{r}\tilde{y}_{e^{\prime
}}\right\Vert =\left\vert \tilde{r}\right\vert \left\Vert \tilde{x}_{e}-%
\tilde{y}_{e^{\prime }}\right\Vert =\left\vert \tilde{r}\right\vert \tilde{d}%
(\tilde{x}_{e},\tilde{y}_{e^{\prime }}).
\end{equation*}%
Suppose that the conditions of the proposition are satisfied $.$ Taking $%
\left\Vert \tilde{x}_{e}\right\Vert =\tilde{d}(\tilde{x}_{e},\tilde{\theta}%
_{0})$ \ for every $\tilde{x}_{e}\tilde{\in}SV(\tilde{X})$ we have

\begin{enumerate}
\item[(N1).] $\left\Vert \tilde{x}_{e}\right\Vert =\tilde{d}(\tilde{x}_{e},%
\tilde{\theta}_{0})\tilde{\geq}\tilde{0}$ and $\left\Vert \tilde{x}%
_{e}\right\Vert =\tilde{d}(\tilde{x}_{e},\tilde{\theta}_{0})=\tilde{0}%
\Leftrightarrow \tilde{x}_{e}=\tilde{\theta}_{0};$

\item[(N2).] $\left\Vert \tilde{r}\tilde{x}_{e}\right\Vert =\tilde{d}(\tilde{%
r}\tilde{x}_{e},\tilde{\theta}_{0})=\tilde{d}(\tilde{r}\tilde{x}_{e},\tilde{r%
}.\tilde{\theta}_{0})=\left\vert \tilde{r}\right\vert \tilde{d}(\tilde{x}%
_{e},\tilde{\theta}_{0})=\left\vert \tilde{r}\right\vert \left\Vert \tilde{x}%
_{e}\right\Vert ;$

\item[(N3).] 
\begin{eqnarray*}
\left\Vert \tilde{x}_{e}+\tilde{y}_{e^{\prime }}\right\Vert &=&\tilde{d}(%
\tilde{x}_{e}+\tilde{y}_{e^{\prime }},\tilde{\theta}_{0})=\tilde{d}(\tilde{x}%
_{e},-\tilde{y}_{e^{\prime }}) \\
&\tilde{\leq}&\tilde{d}(\tilde{x}_{e},\tilde{\theta}_{0})+\tilde{d}(\tilde{%
\theta}_{0},-\tilde{y}_{e^{\prime }}) \\
&=&\left\Vert \tilde{x}_{e}\right\Vert +\left\vert -\tilde{1}\right\vert
\left\Vert \tilde{y}_{e^{\prime }}\right\Vert =\left\Vert \tilde{x}%
_{e}\right\Vert +\left\Vert \tilde{y}_{e^{\prime }}\right\Vert .
\end{eqnarray*}
\end{enumerate}
\end{proof}

\begin{definition}
Let $T:SV(\tilde{X})\rightarrow SV(\tilde{Y})$ be a soft mapping. Then $T$
is said to be soft linear operator if
\end{definition}

\begin{enumerate}
\item[(L1).] $T$\textit{\ is additive, i.e.,}$T(\tilde{x}_{e}+\tilde{y}%
_{e^{\prime }})=T(\tilde{x}_{e})+T(\tilde{y}_{e^{\prime }})$\textit{\ for
every }$\tilde{x}_{e},\tilde{y}_{e^{\prime }}\tilde{\in}SV(\tilde{X}),$

\item[(L2).] $T$\textit{\ is homogeneous, i.e., for every soft scalar }$%
\tilde{r},$\textit{\ }$T(\tilde{r}\tilde{x}_{e})=\tilde{r}$\textit{.}$T(%
\tilde{x}_{e})$\textit{\ for every }$\tilde{x}_{e}\tilde{\in}SV(\tilde{X}),$
\end{enumerate}

\begin{definition}
The soft operator $T:SV(\tilde{X})\rightarrow SV(\tilde{Y})$ is said to be
soft continuous at $\tilde{x}_{e_{0}}^{0}\tilde{\in}SV(\tilde{X})$\ if for
every sequence $\left\{ \tilde{x}_{e_{n}}^{n}\right\} $ of soft vectors of $%
\tilde{X}$ with $\tilde{x}_{e_{n}}^{n}\rightarrow \tilde{x}_{e_{0}}^{0}$ as $%
n\rightarrow \infty ,$ we have $T(\tilde{x}_{e_{n}}^{n})\rightarrow T(\tilde{%
x}_{e_{0}}^{0})$ as $n\rightarrow \infty $. If $T$ is soft continuous at
each soft vector of $SV(\tilde{X}),$ then $T$ is said to be soft continuous
operator.
\end{definition}

\begin{definition}
The soft operator $T:SV(\tilde{X})\rightarrow SV(\tilde{Y})$ is said to be
soft bounded,\ if there exists a soft real number $\tilde{M}$ such that 
\begin{equation*}
\left\Vert T(\tilde{x}_{e})\right\Vert \tilde{\leq}\tilde{M}\left\Vert 
\tilde{x}_{e}\right\Vert ,
\end{equation*}%
for all $\tilde{x}_{e}\tilde{\in}SV(\tilde{X}).$
\end{definition}

\begin{theorem}
The soft operator $T:SV(\tilde{X})\rightarrow SV(\tilde{Y})$ is soft
continuous if and only if it is soft bounded.
\end{theorem}

\begin{proof}
Assume that $T:SV(\tilde{X})\rightarrow SV(\tilde{Y})$ be soft continuous
and $T$ is not soft bounded. Thus, there exists at least one sequence $%
\left\{ \tilde{x}_{e_{n}}^{n}\right\} $ such that%
\begin{equation}
\left\Vert T(\tilde{x}_{e_{n}}^{n})\right\Vert \tilde{\geq}\tilde{n}%
\left\Vert \tilde{x}_{e_{n}}^{n}\right\Vert ,  \label{1}
\end{equation}%
where $\tilde{n}$ is a soft real number. It is clear that $\tilde{x}%
_{e_{n}}^{n}\neq \tilde{\theta}_{0}$. Let us construct a soft sequence as
follows; 
\begin{equation*}
\tilde{y}_{e_{n}}^{n}=\frac{\tilde{x}_{e_{n}}^{n}}{\tilde{n}\left\Vert 
\tilde{x}_{e_{n}}^{n}\right\Vert }.
\end{equation*}%
It is clear that $\tilde{y}_{e_{n}}^{n}\rightarrow \tilde{\theta}_{0}$ as $%
n\rightarrow \infty .$ Since $T$ is soft continuous, then we have $%
\left\Vert T(\tilde{y}_{e_{n}}^{n})\right\Vert \rightarrow \tilde{0}$ as $%
n\rightarrow \infty .$%
\begin{equation*}
\left\Vert T(\tilde{y}_{e_{n}}^{n})\right\Vert =\left\Vert T\frac{\tilde{x}%
_{e_{n}}^{n}}{\tilde{n}\left\Vert \tilde{x}_{e_{n}}^{n}\right\Vert }%
\right\Vert =\frac{\tilde{1}}{\tilde{n}\left\Vert \tilde{x}%
_{e_{n}}^{n}\right\Vert }\left\Vert T(\tilde{x}_{e_{n}}^{n})\right\Vert 
\tilde{>}\frac{\tilde{n}\left\Vert \tilde{x}_{e_{n}}^{n}\right\Vert }{\tilde{%
n}\left\Vert \tilde{x}_{e_{n}}^{n}\right\Vert }=\tilde{1},
\end{equation*}

which is a contradiction.

Conversely, suppose that $T:SV(\tilde{X})\rightarrow SV(\tilde{Y})$ is soft
bounded and the soft sequence $\left\{ \tilde{x}_{e_{n}}^{n}\right\} $ \ is
convergent to the $\tilde{x}_{e_{0}}^{0}.$ In this case, 
\begin{equation*}
\left\Vert T(\tilde{x}_{e_{n}}^{n})-T(\tilde{x}_{e_{0}}^{0})\right\Vert
=\left\Vert T(\tilde{x}_{e_{n}}^{n}-\tilde{x}_{e_{0}}^{0})\right\Vert \tilde{%
\leq}\tilde{M}\left\Vert \tilde{x}_{e_{n}}^{n}-\tilde{x}_{e_{0}}^{0}\right%
\Vert \rightarrow \tilde{0}
\end{equation*}%
which indicates that $T$ is soft continuous.
\end{proof}

\begin{definition}
Let $T:SV(\tilde{X})\rightarrow SV(\tilde{Y})$ be a soft continuous operator 
$.$ 
\begin{equation*}
\left\Vert T\right\Vert =\inf \left\{ \tilde{M}:\left\Vert T(\tilde{x}%
_{e})\right\Vert \tilde{\leq}\tilde{M}\left\Vert \tilde{x}_{e}\right\Vert
\right\}
\end{equation*}%
is said to be a norm of $T.$
\end{definition}

It is obvious that $\left\Vert T(\tilde{x}_{e})\right\Vert \tilde{\leq}%
\left\Vert T\right\Vert \left\Vert \tilde{x}_{e}\right\Vert .$

\begin{theorem}
Let $T:SV(\tilde{X})\rightarrow SV(\tilde{Y})$ be a soft operator$.$ Then,%
\begin{equation*}
\left\Vert T\right\Vert =\underset{\tilde{x}_{e}\neq \tilde{\theta}_{0}}{%
\sup }\frac{\left\Vert T(\tilde{x}_{e})\right\Vert }{\left\Vert \tilde{x}%
_{e}\right\Vert }=\underset{\left\Vert \tilde{x}_{e}\right\Vert \tilde{\leq}1%
}{\sup }\left\Vert T(\tilde{x}_{e})\right\Vert
\end{equation*}
\end{theorem}

\begin{proof}
Since $\left\Vert T(\tilde{x}_{e})\right\Vert \tilde{\leq}\left\Vert
T\right\Vert \left\Vert \tilde{x}_{e}\right\Vert ,$ we have 
\begin{equation*}
\underset{\left\Vert \tilde{x}_{e}\right\Vert \tilde{\leq}1}{\sup }%
\left\Vert T(\tilde{x}_{e})\right\Vert \tilde{\leq}\underset{\left\Vert 
\tilde{x}_{e}\right\Vert \tilde{\leq}1}{\sup }\left\Vert T\right\Vert
\left\Vert \tilde{x}_{e}\right\Vert \tilde{\leq}\left\Vert T\right\Vert .
\end{equation*}%
Thus, 
\begin{equation}
\underset{\left\Vert \tilde{x}_{e}\right\Vert \tilde{\leq}1}{\sup }%
\left\Vert T(\tilde{x}_{e})\right\Vert \tilde{\leq}\left\Vert T\right\Vert
\label{2}
\end{equation}

On the other hand, from the definition of $\left\Vert T\right\Vert ,$ $%
\forall \tilde{\varepsilon}>\tilde{0},$ $\exists \tilde{x}_{e}$ such that

\begin{equation}
\left\Vert T(\tilde{x}_{e})\right\Vert \tilde{>}\left( \left\Vert
T\right\Vert -\tilde{\varepsilon}\right) \left\Vert \tilde{x}_{e}\right\Vert
\label{3}
\end{equation}

let us take $\tilde{y}_{e}=\frac{\tilde{x}_{e}}{\left\Vert \tilde{x}%
_{e}\right\Vert }$, where $\tilde{x}_{e}\neq \tilde{\theta}_{0}.$ In this
case, we have $\left\Vert \tilde{y}_{e}\right\Vert =\tilde{1}.$ If we write $%
\left\Vert \tilde{x}_{e}\right\Vert \tilde{y}_{e}$ instead of $\tilde{x}_{e}$
in (\ref{3}) we have 
\begin{equation*}
\left\Vert T(\tilde{y}_{e}\left\Vert \tilde{x}_{e}\right\Vert )\right\Vert 
\tilde{>}\left( \left\Vert T\right\Vert -\tilde{\varepsilon}\right)
\left\Vert \tilde{y}_{e}\left\Vert \tilde{x}_{e}\right\Vert \right\Vert
\Rightarrow
\end{equation*}

\begin{equation}
\left\Vert T(\tilde{y}_{e})\right\Vert \tilde{>}\left( \left\Vert
T\right\Vert -\tilde{\varepsilon}\right) \left\Vert \tilde{y}_{e}\right\Vert
=\left( \left\Vert T\right\Vert -\tilde{\varepsilon}\right) .  \label{4}
\end{equation}

If we write $\tilde{x}_{e}$ instead of $\tilde{y}_{e}$ in (\ref{4}), then%
\begin{equation*}
\underset{\left\Vert \tilde{x}_{e}\right\Vert \tilde{\leq}1}{\sup }%
\left\Vert T(\tilde{x}_{e})\right\Vert \tilde{>}\left\Vert T\right\Vert -%
\tilde{\varepsilon}.
\end{equation*}%
Since $\tilde{\varepsilon}$ is arbitary we have%
\begin{equation}
\underset{\left\Vert \tilde{x}_{e}\right\Vert \tilde{\leq}1}{\sup }%
\left\Vert T(\tilde{x}_{e})\right\Vert \tilde{\geq}\left\Vert T\right\Vert .
\label{5}
\end{equation}

From (\ref{2}) and (\ref{5}), we have 
\begin{equation*}
\left\Vert T\right\Vert =\underset{\tilde{x}_{e}\neq \tilde{\theta}_{0}}{%
\sup }\frac{\left\Vert T(\tilde{x}_{e})\right\Vert }{\left\Vert \tilde{x}%
_{e}\right\Vert }=\underset{\left\Vert \tilde{x}_{e}\right\Vert \tilde{\leq}1%
}{\sup }\left\Vert T(\tilde{x}_{e})\right\Vert .
\end{equation*}
\end{proof}

\begin{theorem}
Let $T:SV(\tilde{X})\rightarrow SV(\tilde{Y})$ be a soft operator then $%
\left\Vert T\right\Vert $is a soft norm.
\end{theorem}

\begin{proof}

\begin{enumerate}
\item[(N1)] $\left\Vert T\right\Vert \tilde{\geq}\tilde{0}.$ If $\left\Vert
T\right\Vert =\tilde{0},$ then for all $\tilde{x}_{e}\tilde{\in}SV(\tilde{X}%
) $ we have $T(\tilde{x}_{e})=\tilde{\theta}_{0}$ so that $T=\theta .$

\item[(N2)] $\left\Vert \tilde{r}.T\right\Vert =\underset{\left\Vert \tilde{x%
}_{e}\right\Vert =\tilde{1}}{\sup }\left\Vert \tilde{r}.T(\tilde{x}%
_{e})\right\Vert =\underset{\left\Vert \tilde{x}_{e}\right\Vert =\tilde{1}}{%
\sup }\left\vert \tilde{r}\right\vert \left\Vert T(\tilde{x}_{e})\right\Vert
=\left\vert \tilde{r}\right\vert \underset{\left\Vert \tilde{x}%
_{e}\right\Vert =\tilde{1}}{\sup }\left\Vert T(\tilde{x}_{e})\right\Vert =%
\tilde{r}.\left\Vert T\right\Vert $

\item[(N3)] 
\begin{eqnarray*}
\left\Vert T+S\right\Vert &=&\underset{\left\Vert \tilde{x}_{e}\right\Vert 
\tilde{\leq}\tilde{1}}{\sup }\left\Vert \left( T+S\right) (\tilde{x}%
_{e})\right\Vert =\underset{\left\Vert \tilde{x}_{e}\right\Vert \tilde{\leq}%
\tilde{1}}{\sup }\left\Vert T(\tilde{x}_{e})+S(\tilde{x}_{e})\right\Vert \\
&\tilde{\leq}&\underset{\left\Vert \tilde{x}_{e}\right\Vert \tilde{\leq}%
\tilde{1}}{\sup }\left\Vert T(\tilde{x}_{e})\right\Vert +\underset{%
\left\Vert \tilde{x}_{e}\right\Vert \tilde{\leq}\tilde{1}}{\sup }\left\Vert
S(\tilde{x}_{e})\right\Vert =\left\Vert T\right\Vert +\left\Vert
S\right\Vert .
\end{eqnarray*}
\end{enumerate}
\end{proof}

\begin{theorem}
Let $T:SV(\tilde{X})\rightarrow SV(\tilde{Y})$ \ and $S:SV(\tilde{Y}%
)\rightarrow SV(\tilde{Z})$ be two soft operators$.$Then
\end{theorem}

\begin{enumerate}
\item[a)] 
\begin{equation*}
\left\Vert S\circ T\right\Vert \tilde{\leq}\left\Vert S\right\Vert
\left\Vert T\right\Vert ;
\end{equation*}

\item[b)] If $T:SV(\tilde{X})\rightarrow SV(\tilde{X})$ is a soft operator,
then 
\begin{equation*}
\left\Vert T^{n}\right\Vert \tilde{\leq}\left\Vert T\right\Vert ^{n}.
\end{equation*}%
is satisfied.
\end{enumerate}

\begin{proof}

\begin{enumerate}
\item[a)] 
\begin{eqnarray*}
\left\Vert S\circ T\right\Vert &=&\sup \left\{ \left\Vert \left( S\circ
T\right) \left( \tilde{x}_{e}\right) \right\Vert :\left\Vert \tilde{x}%
_{e}\right\Vert \tilde{\leq}\tilde{1}\right\} \\
&=&\sup \left\{ \left\Vert S(T\left( \tilde{x}_{e}\right) )\right\Vert
:\left\Vert \tilde{x}_{e}\right\Vert \tilde{\leq}\tilde{1}\right\} \\
&\tilde{\leq}&\sup \left\{ \left\Vert S\right\Vert .\left\Vert T\left( 
\tilde{x}_{e}\right) \right\Vert :\left\Vert \tilde{x}_{e}\right\Vert \tilde{%
\leq}\tilde{1}\right\} \\
&\tilde{\leq}&\left\Vert S\right\Vert \left\Vert T\right\Vert
\end{eqnarray*}

\item[b)] If we take $T=S$ then we have $\left\Vert T^{2}\right\Vert \tilde{%
\leq}\left\Vert T\right\Vert ^{2}$. Then $\left\Vert T^{n}\right\Vert \tilde{%
\leq}\left\Vert T\right\Vert ^{n}$ is obtained \ 
\end{enumerate}
\end{proof}

\section{\textbf{\ }\noindent}


\begin{thebibliography}{99}
\bibitem{mali} M.I. Ali, F. Feng, X. Liu, W.K. Min and M. Shabir, On some
new operations in soft set theory, Comput. Math. Appl.49(2005) 1547-1553.

\bibitem{bayram} S. Bayramov, C. Gunduz(Aras), Soft locally compact and soft
paracompact spaces, Journal of Mathematics and System Science, 3(2013)
122-130

\bibitem{donlin} S. Das, P. Majumdar, S.K. Samanta, On soft lineer spaces
and soft normed lineer spaces, arXiv:1308.1016 [math.GM].

\bibitem{dlin} S. Das and S. K. Samanta, Soft linear operators in soft
normed linear spaces, Annals of Fuzzy Mathematics and Informatics,6(2)
(2013) 295-314.

\bibitem{sam1} S. Das and S. K. Samanta, Soft real sets, soft real numbers
and their properties, J. Fuzzy Math. 20 (3) (2012) 551-576.

\bibitem{sam2} S. Das and S. K. Samanta, Soft metric, Annals of Fuzzy
Mathematics and Informatics, 6(1) (2013) 77-94.

\bibitem{maji2} P.K.Maji, R.Biswas, A.R.Roy, Soft set theory, Comput. Math.
Appl.45 (2003) 555-562.

\bibitem{molod} D. Molodtsov, Soft set-theory-first results, Comput. Math.
Appl.37(1999) 19-31.

\bibitem{inner} S. Das, S. K. Samanta, On soft inner product spaces, Annals
of Fuzzy Mathematics and Informatics, 6(1) (2013) 151-170.

\bibitem{element} S. Das and S. K. Samanta, On soft metric spaces, J. Fuzzy
Math. accepted.
\end{thebibliography}
\end{document}